\documentclass[a4paper,11pt]{amsart}
\usepackage[utf8]{inputenc}
\usepackage[english]{babel}
\usepackage{mathrsfs}
\usepackage{pxfonts}
\usepackage{enumerate}
\usepackage{color}
\usepackage{todonotes}
\usepackage{tikz}
\usetikzlibrary{arrows, babel}
\usepackage{tikz-cd}
\usepackage{hyperref}
\usepackage[capitalise,noabbrev,nameinlink]{cleveref}

\theoremstyle{plain}
\newtheorem{theorem}{Theorem}[section]
\theoremstyle{plain}
\newtheorem{prop}[theorem]{Proposition}
\theoremstyle{plain}
\newtheorem{coro}[theorem]{Corollary}
\theoremstyle{plain}
\newtheorem{lemma}[theorem]{Lemma}
\theoremstyle{definition}
\newtheorem{defi}[theorem]{Definition}
\theoremstyle{definition}

\theoremstyle{definition}
\newtheorem{example}[theorem]{Example}
\newtheorem{recollection}[theorem]{Recollection}
\theoremstyle{remark}
\newtheorem{remark}[theorem]{Remark}
\theoremstyle{definition}

\newcommand{\sbz}[1]{\Sigma B\mathbb{Z}/#1}
\newcommand{\bz}[1]{B\mathbb{Z}/#1}

\newcommand{\z}[1]{\mathbb{Z} #1}

\newcommand{\map}{\mathrm{map}}

\newcommand{\Q}{\mathbb{Q}}
\newcommand{\Hom}{\operatorname{Hom}}
\newcommand{\p}[2]{#1^{\wedge}_{#2}}
\newcommand{\mapp}{\mbox{map}_{\ast}}

\newcommand{\ev}{\operatorname{ev}}

\newcommand{\Ima}{\operatorname{Im}}

\newcommand{\holim}[1]{\operatorname{holim}_{#1}}
\newcommand{\F}{\mathcal{F}}
\newcommand{\ob}{\operatorname{Ob}}
\newcommand{\Inj}{\operatorname{Inj}}
\newcommand{\Aut}{\operatorname{Aut}}
\newcommand{\Inn}{\operatorname{Inn}}
\newcommand{\Out}{\operatorname{Out}}
\newcommand{\Lin}{\mathcal{L}}
\newcommand{\Linp}{B\F}
\newcommand{\Mor}{\operatorname{Mor}}
\newcommand{\Syl}{\operatorname{Syl}}
\newcommand{\Res}{\operatorname{Res}}
\newcommand{\Ind}{\operatorname{Ind}}
\newcommand{\Iso}{\operatorname{Iso}}
\newcommand{\OF}{\mathcal{O(F)}}
\newcommand{\OFc}{\mathcal{O}(\mathcal{F}^{c})}
\newcommand{\OFp}{\mathcal{O}^{p}_{\mathcal{F}}(S)}
\newcommand{\Rep}{\operatorname{Rep}}
\newcommand{\Topo}{\operatorname{Top}}
\newcommand{\hocolim}{\operatorname{hocolim}}

\newcommand{\sz}{\operatorname{Sz}}
\newcommand{\bsz}{B \! \operatorname{Sz}}

\DeclareRobustCommand\longtwoheadrightarrow
{\relbar\joinrel\twoheadrightarrow}

\title{Cellular approximations of fusion systems}
\author[G. Carrión Santiago]{Guille Carrión Santiago}
\address{Departamento de Matemática Aplicada, Ciencia e Ingeniería de los Materiales y Tecnología\linebreak Electrónica, ESCET, Universidad Rey Juan Carlos, 28933 Móstoles (Madrid), Spain}
\email{Guille.CarrionS@urjc.es}
\author[N. Castellana]{Natàlia Castellana}
\address{Departament de Matemàtiques, Universitat Autònoma de Barcelona, E-08193 Bellaterra, Spain}
\email{Natalia.Castellana@uab.cat}
\author[A. Gavira-Romero]{Alberto Gavira-Romero}

\date{}

\begin{document}
\begin{abstract}
In this paper we study the cellularization of classifying spaces of saturated fusion systems with respect to classifying spaces of finite $p$-groups. We give an explicit criteria to decide when a classifying space is cellular and we explicitly compute the cellularization for a family of exotic examples.
\end{abstract}
\maketitle
\section*{Introduction}
In the 1990s, E. Dror-Farjoun and W. Chach\'olski generalized the concept of $CW$-complex, spaces built from spheres by means of pointed homotopy colimits. Let $A$ be a pointed space and 
let $\mathcal{C}(A)$ denote the smallest collection of pointed spaces that contains $A$ and it is closed under weak equivalences and pointed homotopy colimits. A pointed space $X$ is
\textit{$A$-cellular} if $X \in \mathcal{C}(A)$. Later, this notion was extended to stable homotopy theory. In a similar way, the same notion is considered in stable homotopy theory.
\medskip

Let $G$ be a finite group, $p$ a prime such that $p\mid {|G|}$ and $S\leq G$ a Sylow $p$-subgroup. The fact that $\p{BG}{p}$ is a stable retract of $BS$ implies that $\Sigma^\infty_+\p{BG}{p}$ belongs to $\mathcal C(\Sigma^\infty_+BS)$.
In unstable homotopy theory, $\p{BG}{p}$ is not a retract of $BS$, but we can ask ourselves whether $\p{BG}{p}$ is in the cellular class $\mathcal C(BS)$. More generally, given a finite $p$-group $P$, when $\p{BG}{p}$ belongs to $\mathcal C(BP)$? 
Previous works in finite groups (see \cite{MR2272149}, \cite{MR2351607} and \cite{MR2823972}) suggest the strong relationship between the cellularity properties of $\p{BG}{p}$ with respect to classifying spaces of finite $p$-groups and the fusion structure of $G$ at the prime $p$.\medskip

The homotopy type of $\p{BG}{p}$ is determined by its $p$-local structure, the fusion system associated to $G$. An abstract notion of \emph{fusion system} over a finite $p$-group $S$ was developed by L. Puig in an unpublished work. Afterwards, D. Benson suggested the idea of associating a 
classifying space to each saturated fusion system (see \cite{Benson}); which was formulated and developed by C. Broto, R. Levi and B. Oliver
in \cite{MR1992826}. The existence and uniqueness of the classifying space were proved by A. Chermak 
\cite{MR3118305}, see also Oliver \cite{Oliver-ExistenceL}. Previously, K. Ragnarsson \cite{MR2199459} constructed a classifying space spectrum $\mathbb B\F$ associated to $\F$ by splitting the spectrum of $BS$ via an idempotent stable selfmap. Analogously to the situation for finite groups, we have then $\mathbb B\F\in \mathcal C({\Sigma_+^\infty}BS)$. \medskip

In this work, we address the question of when $B\F$, the classifying space of $\F$, is $BP$-cellular, and we obtain the following result, which the main theorem of the paper.

\medskip

\noindent{\bf \cref{t:BP << Lp}.}
\noindent {\it  Let $(S,\F)$ be a saturated fusion system and let $P$ be a finite $p$-group. Then $\Linp$ is $BP$-cellular if and only if $S = Cl_{\F}(P)$, where $Cl_\F(P)$ is the smallest strongly $\F$-close subgroup of $S$ which contains all the images of homomorphisms $P\to S$
}
\medskip

From this theorem we obtain the following consequences.
\medskip

\noindent{\bf \cref{c:BS << Lp}.}
\noindent {\it 
Let $(S,\F)$ be a saturated fusion system and let $P$ be a finite $p$-group.
\begin{enumerate}[(a)]
\item The classifying space $\Linp$ is $BS$-cellular.
\item Let $(S,\F')$ be a saturated fusion system with $\F\subset \F'$. If $B\F$ is $BP$-cellular then $B\F'$ is also $BP$-cellular. In particular, if $BS$ is $BP$-cellular then so is $\Linp$.
\item Let $Q\twoheadrightarrow P$ be an epimorphism of finite $p$-groups. If $\Linp$ is $BP$-cellular, then it is also $BQ$-cellular. 
\item Let $A$ be a pointed connected space. If $Cl_\F((\pi_1 A)_{ab})=S$, then $\Linp$ is $A$-cellular.
\item Let $\Omega_{p^m}(S)$ be the (normal) subgroup of $S$ generated by its elements of order $p^i$, with $i \leq m$. Then $\Linp$ is $\bz{p^m}$-cellular if and only if 
$S = Cl_{\F}(\Omega_{p^m}(S))$. In particular, there is a non-negative integer $m_0 \geq 0$ such that $\Linp$ is $\bz{p^m}$-cellular for all $m \geq m_0$.
\end{enumerate}
}
\medskip

The key concept to approach \cref{t:BP << Lp} is the kernel of a map introduced by  D. Notbohm \cite{MR1286829}. This construction allows us to characterize null-homotopic maps:
\medskip

\noindent{\bf \cref{t:Dwyer p local}.}
\noindent {\it Let $(S,\F)$ be a saturated fusion system, and $Z$ be a connected $p$-complete \mbox{$\sbz{p}$-null} space. A map $f \colon \Linp \rightarrow Z$ is 
null-homotopic if and only if $\ker(f) = S$.
}\medskip

Moreover, given a map $f \colon B\F \rightarrow Z$, we show that $\ker(f)$ is a strongly $\F$-closed subgroup of $S$. Furthermore, any strongly $\F$-closed subgroup $K\leq S$ is the kernel of a map.

The last two sections are devoted to give explicit examples. One special case is when $Cl_{\F}(S)$ is normal in $\F$. 
\medskip

\noindent{\bf \cref{c:normal}.}
\noindent {\it Let $(S,\F)$ be a fusion system and let $P$ be a finite $p$-group. If $Cl_{\F}(P) \vartriangleleft \F$, then $CW_{BP} (\Linp)$ is homotopy equivalent to the homotopy
fiber of $\Linp \rightarrow B(\F/Cl_{\F}(P))$.
}
\medskip

This result allow us to compute, for all $r \geq 1$, the $\bz{p^r}$-cellularization of the classifying space of $\z{/p^n} \wr \z{/q}$, with $p \neq q$, and of the Suzuki group 
$\sz(2^n)$, with $n$ an odd integer at least $3$ (Example~\ref{e:NG(S) controls fus}).

The last section contains an explicit description of the $\bz{3^l}$-cellularization of the classifying space of a family of exotic fusion systems over a finite $3$-group given in \cite{MR2272147}. \\[-.25cm]

\noindent{\bf \cref{c:exotic example}.}
\noindent {\it Let $\F$ be an exotic fusion system over $B(3,r;0,\gamma,0)$ such that $\F$ has at least one $\F$-Alperin rank two elementary abelian $3$-subgroup given in \cite[Theorem 5.10]{MR2272147}.
Then 
\begin{enumerate}[(i)]
\item If $\gamma = 0$, then $\Linp$ is $\bz{3^l}$-cellular for all $l \geq 1$.
\item If $\gamma \neq 0$, then $\Linp$ is $\bz{3^l}$-cellular if and only if $l \geq 2$.
\end{enumerate}
}
\medskip

Moreover, when $\gamma \neq 0$ and $l=1$, $Cl_{\F}(\z{/3})=\langle s,s_2\rangle$ and $CW_{\bz3}(B\F)$ is the homotopy fiber of a map $B\F\rightarrow \p{(B\Sigma_3)}{3}$.
\medskip

{\bf Outline of this paper.} In \cref{s:fus sys} we review the notion of the fusion system and its homotopy properties needed for the rest of the paper. \cref{s: Homotopy properties of cellular spaces} describes we review homotopy properties of cellular spaces. The notion of the kernel of a map and its properties are studied in \cref{s:ker f}. In \cref{s:cw BF}, we prove the main theorem and its corollaries. Finally, in \cref{s:exotic example} we give a complete description of the cellularization of classifying spaces for a family of exotic fusion system\medskip

\noindent
{\bf Acknowledgments.} 

The first author wishes to thank the Department of Mathematics at the Universitat Autònoma de Barcelona for its hospitality. Their work was supported by the Universidad de Málaga under grant G RYC-2010-05663, as well as by the Spanish Ministry of Science and Innovation (MICINN) grant PID2023-149804NB-I00.
The first two authors are partially funded by the Comissionat per Universitats i Recerca de la Generalitat de Catalunya (grant No. 2021-SGR-01015). The second author is partially founded by MICINN grant  PID2024-158573NB-I00.
The third author is also partially supported by the Xunta de Galicia through the Proxecto Emerxente EM 2013/016.\medskip

\section{The homotopy theory of fusion systems}\label{s:fus sys}
In this section we introduce the definition and homotopical properties of fusion systems needed for the rest of this paper. Briefly, a saturated fusion system is a small subcategory of the category of groups which encodes fusion/conjugacy data between subgroups of a fixed finite $p$-group $S$, as formalized by L. Puig (see \cite{MR2253665}, \cite{MR2848834}). Such a category has a notion of classifying space (which is not the nerve of the category) introduced by C. Broto, R. Levi and B. Oliver (see \cite{MR1992826}) which satisfy many of the rigid homotopy theoretic properties of $p$-completed classifying spaces of finite groups.

\begin{defi}
Let $S$ be a finite $p$-group. A \textit{saturated fusion system on $S$} is a subcategory $\F$ of the category of groups with $\ob(\F)$  the set of all subgroups of $S$ and 
such that it satisfies the following properties. For all $P,Q \leq S$:
\begin{enumerate}[($f.1$)]
\item $\Hom_S (P,Q) \subset \Hom_{\F} (P,Q) \subset \Inj(P,Q)$.
\item Each $\varphi \in \Hom_{\F} (P,Q)$ is the composite of an isomorphism in $\F$ followed by an inclusion.
\end{enumerate}
\begin{enumerate}[$(s.1)$]
\item For all $P \leq S$ which is \textit{fully normalized in $\F$}, $P$ is also \textit{fully centralized in $\F$} (i.e., if $|N_S(P)|\ge |N_S(P')|$, for all $P'$ that is $\F$-conjugate to $P$, then $|C_S(P)| \geq |C_S(P')|$), and 
$\Out_S(P) \in \Syl_p(\Out_{\F}(P))$.
\item Let $P \leq S$ and $\varphi \in \Hom_{\F}(P,S)$ be such that $\varphi(P)$ is fully centralized. If we set
$$
N_{\varphi} = \{g \in N_S(P) \mid \varphi c_g \varphi^{-1} \in \Aut_S(\varphi(P))\},
$$
then there is $\bar{\varphi} \in \Hom_{\F} (N_{\varphi},S)$ such that $\bar{\varphi}|_P = \varphi$.
\end{enumerate}
\end{defi}

\begin{example}
Given a finite group $G$ with a fixed Sylow $p$-subgroup $S$, let $\F_S(G)$ be the category with $\Mor_{\F_S(G)} (P,Q) = \Hom_G (P,Q)$ for all $P,Q \leq S$, where 
\[
\Hom_G (P,Q) = \{\varphi \in \Hom(P,Q) \mid \varphi = c_g \mbox{ for some } g \in G\}.
\]
This category $\F_S(G)$ is a saturated fusion system
(see \cite[Proposition 1.3]{MR1992826}).
\end{example}

\begin{recollection}
Let $R$ be a commutative ring. The Bousfield-Kan $R$-completion functor is a homological localization functor when we restrict to $R$-good spaces (see \cite[p. 205]{MR0365573}), that is, those spaces for which the $R$-completion is $R$-complete. When $R=\z/{p}$, we just say that a space is $p$-good. By \cite[Proposition VII.5.1]{MR0365573}, if the fundamental group of a pointed space $X$ is finite, then $X$ is $p$-good for any prime $p$.
\end{recollection}
The homotopy type of the Bousfield-Kan $p$-completion of the classifying space $BG^\wedge_p$ can also be modelled up to $p$-completion by means of a category, the centric linking system $\Lin_S(G)$ introduced by Broto, Levi and Oliver \cite{MR1961340}, which projects onto a full subcategory of the fusion system $\F_S(G)$.

\begin{example}\label{ex:linking_system_G}
Given a finite group $G$ with a fixed Sylow $p$-subgroup $S$, the centric linking system $\Lin_S(G)$ is the category whose objects are $p$-centric subgroups of $G$ with morphisms defined by $\Mor_{\Lin_S(G)} (P,Q) = \{ x \in G \mid xP^{-1}x \leq Q\}/O^p(C_G(P))\}$ for all $P,Q \leq S$. Recall that a $p$-subgroup $P\leq G$ is $p$-centric if $Z(P')$ is a Sylow $p$-subgroup of $C_G(P')$ for all $P'$ that is $G$-conjugate to $P$. 
The $p$-completion of its nerve recovers  the classifying space, that is $\p{|\Lin_S(G)|}{p} \simeq \p{BG}{p}$ (see \cite{MR1961340}). Moreover, the monomorphism $S\rightarrow \Aut_{\mathcal L_S(G)}(S)$ induces a continuous map $BS\rightarrow |\mathcal L_S(G)|$ through the automorphisms of $S$ in $\mathcal L_S(G)$ which realizes, up to $p$-completion, the inclusion of the Sylow $p$-subgroup.
\end{example}  

This notion can be generalized to an abstract fusion system aiming for a definition of the  classifying space of a saturated fusion system. 

\begin{defi}
\label{def:F-centric subgroup}
Let $(S,\mathcal F)$ be a saturated fusion system. We say that a subgroup $P\leq S$ is $\mathcal F$-centric if  $C_S (P') = Z(P')$ for any $P'\leq S$ which is $\mathcal F$-isomorphic to $P$. We denote by $\mathcal F^c$ the full subcategory of $\mathcal{F}$ on $\mathcal{F}$-centric subgroups.
\end{defi}

The centric linking system $\mathcal L$ associated with a saturated fusion system $\F$ is a category whose objects are the $\F$-centric subgroups of $S$, and it is equipped with a functor ${\pi\colon \mathcal L\to \F^c}$ that is the identity on objects and surjective on morphisms, and, for every $\F$-centric subgroup $P\le S$, a ``distinguished'' monomorphism
\[
\delta_P\colon P\to \Aut_{\mathcal{L}}(P).
\]
verifying some extra axioms, see \cite[Definition 1.7]{MR1992826} for details. The $p$-completion of the nerve of $\mathcal L$ will play the role of the classifying space of $\F$, i.e.,
$\Linp=\p{|\mathcal L|}{p}$.

A key result in the theory of fusion systems is the existence and uniqueness (up to equivalence) of centric linking systems associated to a given saturated fusion system.  
The result was proven first by Chermak in \cite{MR3118305} using the new theory of localities; and another proof was given by Oliver \cite{Oliver-ExistenceL} using the obstruction theory developed in \cite{MR1992826}. 

\begin{theorem}[{\cite{MR3118305, Oliver-ExistenceL}}]
\label{theorem: LL-ExistenceL}
Let $(S,\mathcal F)$ be a saturated fusion system. Up to equivalence, there exists a unique centric linking system $\mathcal{L}$ associated to~$\mathcal F$. 
\end{theorem}

\begin{defi}
Let $(S,\mathcal F)$ be a saturated fusion system. The classifying space $B\mathcal F$ is defined to be $\p{|\mathcal L|}{p}$, and the inclusion of the Sylow $p$-subgroup $\Theta \colon BS \rightarrow \Linp$ is $\p{(\delta_S)}{p}$. 
\end{defi}

The classifying space of a fusion system has good homotopical properties with respect to $p$-completion.
\begin{prop}
\label{prop:hmtpy gps}
Let $(S,\F)$ be a saturated fusion system. Then $\Linp$ is a nilpotent $p$-complete space such that $\pi_i(\Linp)$ and $H_i(\Linp;\mathbb Z)$ are finite $p$-groups for all $i\ge 1$. Moreover, $\pi_1(\Linp) \cong S/\OFp$,
where 
$\OFp := \langle [Q,\mathcal{O}^p(\Aut_{\F}(Q))] \mid Q \leq S \rangle$.
\end{prop}

\begin{proof}
The description of the fundamental group $\pi_1(\Linp)$ is given in \cite[Theorem B]{MR2302515}. The fact that $\Linp$ is $p$-complete follows from 
\cite[Proposition I.5.2]{MR0365573} since the nerve of the associated linking system $|\Lin|$ is $p$-good by \cite[Proposition 1.12]{MR1992826}. Finally, $H_i(\Linp;\mathbb Z)$ and $\pi_i(\Linp)$ are finite 
$p$-groups for all $i\geq 1$ by \cite[Lemma 7.6]{MR2518640}. Furthermore, $\Linp$ is nilpotent according to \cite[Proposition VII.4.3(ii)]{MR0365573}.
\end{proof}
\begin{recollection}
Notice that given a saturated fusion system $(S,\F)$, the classifying space $B\F$ is $\Q$-acyclic and $\mathbb F_q$-acyclic, that is, $\tilde{H}^*(\Linp;\Q)=0$ and $\tilde{H}^*(\Linp;\mathbb F_q)=0$ whenever $(q,p)=1$, since $\Sigma_+^\infty \Linp$ is a stable retract of $\Sigma^\infty_+BS$ (see \cite{MR2199459}).
\end{recollection}

The subgroup $\OFp$ introduced in \cref{prop:hmtpy gps} plays an important role in describing finite coverings of classifying spaces of fusion systems. 

\begin{theorem}[{\cite[Theorem 4.4]{MR2302515}}]
\label{thm:quotients of L}    
Let $(S,\mathcal F)$ be a saturated fusion system and $T$ be a $p$-group such that $\mathcal O_{\mathcal F}^p(S)\unlhd T\leq S$. Then there is a saturated fusions system $\mathcal F_T$ on $T$ such that $B\mathcal F_T$ is homotopy equivalent to the covering space of $B\mathcal F$ with fundamental group $T/\mathcal O_{\mathcal F}^p(S)$.
\end{theorem}

One of the main tools in the study of maps between classifying spaces is the existence of homology decompositions indexed in orbit categories.

\begin{defi}
Let $(S,\mathcal F)$ be a saturated fusion system. The \textit{orbit category} of $\F$ is the category $\OF$ whose objects are the subgroups of $S$, and whose 
morphisms are defined by:
\[
\Mor_{\OF}(P,Q) = \Rep_{\F}(P,Q) := \Inn(Q)\backslash \Hom_{\F}(P,Q).
\]
Then, $\OFc$ denotes the full subcategory of $\OF$ whose objects are the $\F$-centric subgroups of $S$. If $\Lin$ is a centric linking system associated to $\F$, denote by 
$\tilde{\pi}$ the composite functor 
\[
\tilde{\pi} \colon \Lin \overset{\pi}\longrightarrow \F^c \longtwoheadrightarrow \OFc.
\]
\end{defi}

The homotopy type of the nerve of a centric linking system can be described as a homotopy colimit over the orbit category.

\begin{prop}{\cite[Proposition 2.2]{MR1992826}}]
\label{p:L = hocolim}
Let $(S,\F)$ be a saturated fusion system and $\Lin$ be the associated linking system,  $\tilde{\pi}\colon \Lin \rightarrow \OFc$ be the projection functor, and 
$\tilde{B} \colon \OFc \rightarrow \Topo$
be the left homotopy Kan extension over $\tilde{\pi}$ of the constant functor $\Lin \overset{\ast}\rightarrow \Topo$. Then $\tilde{B}$ is a lift of the classifying space functor, that is, $\tilde{B}(P)\simeq BP$, and
\begin{equation*}
|\Lin| \simeq \hocolim_{\OFc} \tilde{B}.
\end{equation*}
In particular, $\Linp \simeq \p{(\hocolim_{\OFc} \tilde{B})}{p}$.
\end{prop}  

One feature of the orbit category we will use is the following.

\begin{prop}
\label{p:|O|=punt}
Let $(S,\F)$ be a saturated fusion system, and $M\in \z{}_{(p)}-\operatorname{Mod}$. Then   
\[
{\lim_{\OFc}}^i \underline{M}=0\mbox{ for all }i>0,
\]
where ${\underline{M}\colon\OFc\rightarrow  \z{}_{(p)}-\operatorname{Mod}}$ is the constant functor with value $M$. In particular, $|\OFc|^\wedge_p\simeq \ast$.
\end{prop}
\begin{proof}
Given $P\leq S$ in $\mathcal F^c$, let $M_P \colon \OFc^{\operatorname{op}} \rightarrow \z{}_{(p)}-\operatorname{Mod}$ be the atomic functor defined by $M_P(Q)=0$ if $Q \neq P$, and $M_P(P)=M$.

We claim that if $\lim^i M_P = 0$ for all $P$ and $i > 0$, then $\lim^i \underline M = 0$ for $i > 0$. This reduction follows from the proof of \cite[Theorem 1.10]{MR1320996}: filtering $\underline M$ as a series of extensions of functors $M_P$, one for each object $P$, and using the long exact sequence for higher limits to deduce the claim.

Finally, higher limits of atomic functors have been studied in \cite[Proposition 6.1 (i),(ii)]{MR1154593}. If $p \nmid |\Out_{\F} (P)|$, then 
\[
{\lim_{\OFc}}^i M_P = \begin{cases}
M^{\Out_{\F} (P)} = M & \mbox{if }i = 0, \\    
0 & \mbox{if }i > 0.
\end{cases}
\]
If $p \mid | \Out_{\F} (P)|$, since the action is trivial, then $ \lim^i M_P = 0$ for $i \geq 0$. Therefore $\lim_{\OFc}^i \underline M= 0$ for $i > 0$.
Finally, in the particular case of $M=\mathbb F_p$, we have $H^i(|\OFc|;M)=\lim^i \underline M = 0$ for all $i>0$. Therefore $|\OFc|^\wedge_p\simeq \ast$.
\end{proof}

\section{Homotopy properties of cellular spaces}\label{s: Homotopy properties of cellular spaces}
In this section we recall the basic definitions in the theory of $A$-homotopy and some of the homotopy properties (see \cite{MR1392221}).

\begin{defi}
Let $A$ be a pointed connected space.
\begin{enumerate}
\item  We say that a space $X$ is \emph{$A$-null} if the evaluation map ${\ev\colon \map(A, X)\to X}$ is a weak equivalence. 
\item We say that a space $X$ is \emph{$A$-acyclic} if ${\ev\colon \map(X, Z)\to Z}$ is a weak equivalence for any $A$-null space $Z$.
\item A pointed map $f\colon Y \rightarrow Z$ is an $A$-equivalence if $f_* \colon \mapp(A,Y) \rightarrow \mapp(A,Z)$ is a weak equivalence.
\item We say that a space $X$ is \emph{$A$-cellular} if for any choice of basepoint in X and for every pointed $A$-equivalence $f\colon Y \rightarrow Z$, the induced map of mapping spaces $f_* \colon \mapp(X,Y) \rightarrow \mapp(X,Z)$ is also a weak equivalence.
\end{enumerate}
\end{defi}

\begin{remark}\label{r:null_loop}
If $X$ is a connected space, then $X$ is $A$-null if and only if $\mapp(A, X)$ is weakly contractible. A connected space $X$ is $\Sigma A$-null if and only if $\map(A,X)_c\rightarrow X$ is a weak equivalence, see \cite[3.A.1]{MR1392221}. This follows from the fact that $X$ is connected and $\Sigma A$-null iff $\Omega X$ is $A$-null, and 
\[
\map(A, \Omega X)\simeq \Omega (\map(A,X))\simeq \Omega(\map(A,X)_c),
\]
where $\map(A,X)_c$ denotes the connected component of $\map(A,X)$ that contains the constant map.
\end{remark}

\begin{theorem} Let $A$ be a connected space.
\begin{enumerate}
\item \cite{MR1257059} There is a \emph{nullification functor} $P_A\colon Spaces \to Spaces$ equipped with a natural transformation $\eta_X\colon X\to P_A(X)$ which is initial with respect to maps to $A$-null spaces. Then $X$ is \emph{$A$-null} (resp. \emph{$A$-acyclic}) if and only if $\eta_X \colon X \simeq P_AX$ (resp. $P_AX\simeq \ast$). 

\item \cite{MR1392221} There is a \emph{cellularization functor} $CW_A\colon Spaces_* \to Spaces_*$  equipped with a natural transformation $c_X\colon CW_A(X)\to X$ which is terminal with respect to maps from $A$-cellular spaces. Then $X$ is \emph{$A$-cellular} if and only if $c_X\colon CW_AX\simeq X$. 
\end{enumerate}
\end{theorem}
The nullification and the cellularization functors are related. The $A$-nullification isolates the information of a space $X$ that is not visible by means of $\mapp(A,X)$, the $A$-cellularization describes to what extent a space can be built using $A$ as building blocks.  
If $X$ is $A$-cellular then it is $A$-acyclic, $P_A(X)\simeq *$. If $X$ is $A$-null, then $CW_A(X)$ is contractible, see \cite[3.B.1]{MR1392221}.
Next theorem combines both nullification and cellularization functors to give an explicit description on how the $A$-information of the space $X$ can be treated.

\begin{theorem}[{\cite{MR1408539}}]
\label{t:Wojtek}
Let $A$ and $X$ be connected spaces. Consider $C$ the homotopy cofiber of the evaluation map $ev \colon \bigvee_{[A,X]_{\ast}} A \rightarrow X$. Then, $CW_{A}(X)$ fits in a homotopy fiber sequence as follows:
\[
CW_{A}(X) \rightarrow X\overset r \rightarrow P_{\Sigma A} C.
\]
\end{theorem}

\begin{coro}\label{c:ifrnull_CW}
Let $A$ and $X$ be connected spaces and $C$ be the homotopy cofiber of the evaluation map $ev \colon \bigvee_{[A,X]_{\ast}} A \rightarrow X$. If $r\colon X \rightarrow P_{\Sigma A}C$ is null-homotopic then $X$ is $A$-cellular.
\end{coro}
\begin{proof}
If $r$ is null-homotopic then its homotopy fiber is $CW_AX\simeq X\times \Omega P_{\Sigma A}C\simeq X \times P_{A}(\Omega C)$. Since $CW_AX$ is $A$-acyclic then $P_{A}(\Omega C)\simeq *$. Finally, if $X$ is connected, the same holds for $P_{\Sigma A}C$, and therefore it is also weakly contractible. The map $CW_A(X)\rightarrow X$ is then a weak equivalence.
\end{proof}

Next, we describe a criteria for detecting when a map from an $A$-cellular space is null-homotopic.

\begin{prop}
\label{p:mapCWtoZ}
Let $X$ and $Y$ be pointed connected spaces. Assume that $X$ is $A$-cellular and $Y$ is $\Sigma A$-null. Then a pointed map $f \colon X \rightarrow Y$ is null-homotopic 
if and only if $f_*\colon [A,X]_*\rightarrow [A,Y]_*$ is trivial.
\end{prop}

\begin{proof}
If $f$ is null-homotopic, then $f \circ g$ is null-homotopic for any map $g \colon A \rightarrow X$. Assume now that $f_*\colon [A,X]_*\rightarrow [A,Y]_*$ is trivial, that is, for any
map $g \colon A \rightarrow X$ the composite $f \circ g$ is null-homotopic. 

Let $C$ be the homotopy cofiber of 
$ev \colon \bigvee_{[A,X]_{\ast}} A \rightarrow X$. Consider the Barratt-Puppe long exact sequence of pointed sets
\[
\left[\Sigma\left({\bigvee}_{[A,X]_{\ast}} A\right),Y\right]_* {\rightarrow} [C,Y]_* \stackrel{p^*}{\rightarrow} [X,Y]_* \stackrel{ev^*}{\rightarrow} \left[{\bigvee}_{[A,X]_{\ast}} A,Y\right]_*
\]
If $ev^*([f])=*$, there is $[\tilde{f}]\in [C,Y]_*$, with $p^*(\tilde{f})\simeq f$. Moreover, if $Y$ is $\Sigma A$-null, $[\Sigma(\vee A),Z]_*={*}$, and then $f\simeq *$ if and only if $\tilde{f}\simeq *$. If we prove that $C$ is $\Sigma A$-acyclic, then $\tilde{f}\simeq *$. But this holds since $X$ is $A$-cellular by hypothesis applying Chachólski's fibration in \cref{t:Wojtek}.
\end{proof}

When dealing with classifying spaces of groups or fusion systems, our interest focus on isolating the information which is purely algebraic when $A=BP$ for a finite $p$-group $P$. The following properties are relevant.

\begin{lemma}\label{lem:Bzpnull->BPnull}
Let $P$ be a finite $p$-group. Then,
\begin{enumerate}
\item if $X$ is a $\bz p$-null space, then $X$ is also $BP$-null;
\item $\Linp$ is $\Sigma BP$-null.
\end{enumerate}
\end{lemma}
\begin{proof}
The first statement follows since $BP$ is $\bz p$-acyclic; see \mbox{\cite[Lemma 6.13]{MR1397728}.}
For the second one, by \cite[Theorem 4.4(d)]{MR1992826}, the evaluation $\map(\bz p,\Linp)_c\to\Linp$ is a weak equivalence. Then, by  \cref{r:null_loop}, $\Linp$ is $\Sigma \bz p$-null. Finally, we conclude by the first statement.
\end{proof}
 
The rest of this section is devoted to prove next theorem.

\begin{theorem}
\label{p:pi1finite}
Let $(S,\F)$ be a saturated fusion system and let $P$ be a finite $p$-group. Then $CW_{BP}(\Linp)$ is a nilpotent space whose fundamental group is a finite $p$-group, and there is a homotopy fiber sequence
\[
CW_{BP}(\Linp) \rightarrow \p{CW_{BP}(\Linp)}{p} \rightarrow (\p{CW_{BP}(\Linp)}{p})_{\Q}.
\]
\end{theorem}

To prove \cref{p:pi1finite}, we will use the description given by Chachólski \cite{MR1408539}. But first, we reduce the problem to a saturated fusion system such that the homotopy cofiber of the evaluation map is simply connected.

\begin{prop}
\label{prop:reduction to C 1-connected}
Let $(S,\F)$ be a saturated fusion system and $P$ a finite $p$-group. There exists a saturated fusion system $(S', \F')$ and a $BP$-equivalence $f \colon B\F' \rightarrow \Linp$ such that the homotopy cofiber of $ev' \colon \bigvee_{[BP,B\F']_{\ast}} BP \rightarrow B\F'$ is simply connected.  
\end{prop}

\begin{proof}
The strategy of the proof is to consider an appropriate finite covering of $B\F$. Let $N$ be the normal subgroup of $\pi_1 (\Linp)\cong S/\OFp$ generated by the image of group homomorphisms $P \rightarrow \pi_1(\Linp)$ such that the corresponding pointed map $BP \rightarrow B\pi_1(\Linp)$ lifts to $\Linp$.  That is, $N$ is the normal closure of the image of the group homomorphism induced on fundamental groups by $ev \colon \bigvee_{[BP, \Linp]_{\ast}} BP \rightarrow \Linp$. Applying Seifert-Van Kampen's theorem, we obtain a description of the fundamental group of the cofiber: $\pi_1(C)\cong S/\bar{N}$ where $\OFp\lhd \bar{N}\lhd S$ is the preimage of $N$ in $S$.

By \cref{thm:quotients of L}, there is a saturated fusion system $(S', \F')$ with  $S' \leq S$, $\F' \leq \F$ and a homotopy fiber sequence $\Linp' \rightarrow \Linp \rightarrow B(\pi_1(\Linp) /N)$. In particular, we have that 
\[
\mapp(BP,B\F')\overset{f_*}{\longrightarrow}  \mapp(BP,B\F)_{\{c\}} \overset{pr_{\ast}}{\longrightarrow}  \mapp(BP,B(\pi_1(\Linp) /N))_c
\]
is also a homotopy fibration where $\{c\}$ denotes the set of components of maps which are nullhomotopic when precomposed with $pr \colon \Linp \rightarrow B(\pi_1(\Linp) /N)$. The classical equivalence $\mapp(BP,B(\pi_1(\Linp) /N))\simeq \Hom(P,\pi_1(\Linp) /N)$ shows that the base space is contractible. 

If we check that all components of $\mapp(BP,B\F)$ 
are sent by $pr_{\ast}$ to the component of the constant map, we obtain by definition that $\iota \colon \Linp' \rightarrow \Linp$ is a $BP$-equivalence. Thus, consider a pointed map $h \colon BP \rightarrow B\F$. The theory of covering maps tells us that we need to check that the image of the fundamental group is contained in $N$. But note that the composite $pr \circ h$ is homotopy
equivalent to a map induced by a group homomorphism $\alpha = \pi_1(pr \circ h) \colon P \rightarrow \pi_1(\Linp)$ whose image is in $N$ by construction.

Let $C'$ be the homotopy cofiber of the evaluation map $ev' \colon \bigvee_{[BP, \Linp']_{\ast}} BP \rightarrow \Linp'$. It remains to show that $C'$ is simply connected, that is, the normal closure of the image of ${\pi_1(ev')\colon F\rightarrow \pi_1(\Linp')\cong N}$ is all of $N\leq \pi_1(\Linp)$, where $F$ is the free product of $P$ indexed by $[BP, \Linp']_{\ast}$. Since $N$ is the normal closure of the image $\pi_1(ev)$ in $\pi_1(B\F)$, it is enough to show that for any $f\colon BP \rightarrow \Linp$, we have $\Ima(\pi_1(f))\leq \Ima(\pi_1(ev'))$.  But precisely the equivalence 
$\iota_*\colon \mapp(BP,\Linp')\simeq \mapp(BP,\Linp)$ shows that any such $f\simeq \iota_*(f')$ for $f'\in \mapp(BP,\Linp')$. 
\end{proof}

\begin{remark}\label{r:Pi1ofC}
Let $(S,\F)$ be a saturated fusion system and $P$ a finite $p$-group. Let $C$ be the the homotopy cofiber of ${ev \colon \bigvee_{[BP,B\F]_{\ast}} BP \rightarrow \Linp}$ and $N$ be the normal subgroup of $S$ generated by the images of all morphisms $\Hom(P,S)$. Then, the first paragraph of the proof of \cref{prop:reduction to C 1-connected} shows that  $\pi_1(C)\cong S/(N\OFp)$.
\end{remark}

\begin{proof}[Proof of \cref{p:pi1finite}]
Since $\Linp$ is nilpotent by \cref{prop:hmtpy gps}, we have that $CW_{BP}(\Linp)$ is also nilpotent by \cite[Corollary 3.2]{CASTFLOR}. Moreover, it follows from \cite[Lemma 2.8]{CASTFLOR} that the Bousfield-Kan completion
$R_{\infty} CW_{BP}(\Linp) \simeq \ast$ for $R = \Q$ and $R = \mathbb{F}_q$, $q \neq p$, since $\tilde{H}_{\ast}(BP; R) =0$. Finally, the Sullivan's arithmetic square applied 
to the nilpotent space $CW_{BP}(\Linp)$ gives the homotopy fiber sequence
\[
CW_{BP}(\Linp) \rightarrow \p{CW_{BP}(\Linp)}{p} \rightarrow (\p{CW_{BP}(\Linp)}{p})_{\Q}.
\]

It remains to show that  $\pi_1 CW_{BP}(\Linp)$ is a finite $p$-group. 
By \cref{prop:reduction to C 1-connected} we can assume that the homotopy cofiber of $ev \colon \bigvee_{[BP,\Linp]_{\ast}} BP \rightarrow \Linp$ is simply connected. The Chachólski's homotopy fiber sequence $
CW_{BP}(\Linp) \rightarrow \Linp \rightarrow P_{\Sigma BP} C
$
induces a long  exact sequence of homotopy groups
\[
\ldots \rightarrow \pi_2 (P_{\Sigma BP} C) \rightarrow \pi_1 CW_{BP}(\Linp) \rightarrow \pi_1 (\Linp) \rightarrow 0
\]
where $\pi_1 (\Linp)$ is a finite $p$-group by \cref{prop:hmtpy gps} and $\pi_1(P_{\Sigma BP} C)=0$ (see \cite[Proposition 2.9]{MR1257059}). Therefore, we are reduced to prove that $\pi_2 (P_{\Sigma BP} C)$ is a finite $p$-group which, by Hurewicz's theorem, is isomorphic to $H_2 (P_{\Sigma BP} C; \z{})$
because $P_{\Sigma BP} C$ is simply connected. Moreover, since $\Sigma BP$ is also simply connected,
we obtain an epimorphism $H_2 (C ; \z{}) \twoheadrightarrow H_2(P_{\Sigma BP} C; \z{})$ by \cite[Proposition 3.2]{CellInfEsp}. Summarizing, it is enough to prove that $H_2 (C ; \z{})$ is a finite $p$-group.

The cofibration sequence  
$
\bigvee_{[BP, \Linp]_{\ast}} BP \rightarrow \Linp \rightarrow C
$ 
induces a long exact sequence of homology groups 
\[
\ldots \longrightarrow H_2(\Linp; \z{}) \overset{f_1}\longrightarrow  H_2(C; \z{}) \overset{f_2}\longrightarrow 
H_1 \left({\bigvee}_{[BP, \Linp]_{\ast}} BP; \z{}\right) \longrightarrow\ldots 
\]
where $H_2(\Linp; \z{})$ is a finite $p$-group by \cref{prop:hmtpy gps}. Moreover, 
\[
H_1 \left({\bigvee}_{[BP, \Linp]_{\ast}} BP; \z{}\right) \cong \pi_1\left({\bigvee}_{[BP, \Linp]_{\ast}} BP\right)_{ab} \cong {\bigoplus}_{[BP, \Linp]_{\ast}} P_{ab},
\]
where $[BP,\Linp]_{\ast}$ is a finite set because there is an epimorphism of sets $[BP,BS]_{\ast} \twoheadrightarrow [BP,\Linp]_{\ast}$ by \cite[Theorem 4.4]{MR1992826}, and 
$[BP, BS]_{\ast} \cong \Hom(P, S)$ is finite.  Then, we conclude that $H_2 (C;\z{})$ is a finite $p$-group. 
\end{proof}

\section{The kernel of a map from a classifying space}\label{s:ker f}

In this section, we study the notion of kernel of a map from a classifying space introduced by Notbohm (cf. \cite{MR1286829}), adapted to the context of classifying spaces of fusion systems. The main goal is to show that this subgroup can be used to determine when a map is null-homotopic. Moreover, we characterize the subgroups of the Sylow $p$-subgroup $S$ which are kernels of maps. 

\begin{defi}(\cite{MR1286829})
\label{d:kerf}
Let $(S,\F)$ be a saturated fusion system and $Z$ be a connected $p$-complete $\sbz{p}$-null space. The kernel of a pointed map  $f \colon \Linp \rightarrow Z$ is the subgroup of $S$ defined by:
\[
\ker(f) := \{g \in S \mid f|_{B\langle g \rangle} \simeq \ast \}.
\]
\end{defi}
\begin{recollection}
This notion was introduced by Notbohm in the context of compact Lie groups and discrete subgroups. In \cite[Proposition 2.4]{MR1286829} he shows, with the given definition, that $\ker(f)$ is a subgroup of $S$. In fact, since $S$ is finite, one does not need $Z$ to be $p$-complete, only being $\Sigma B\z/p$-null is necessary. The kernel of a map is homotopy invariant: $\ker(f)=\ker(f')$ if $f\simeq f'$.
\end{recollection}

\begin{theorem}
\label{t:Dwyer p local}
Let $(S,\F)$ be a saturated fusion system, and $Z$ be a connected $p$-complete\linebreak \mbox{$\sbz{p}$-null} space. A map $f \colon \Linp \rightarrow Z$ is 
null-homotopic if and only if $\ker(f) = S$.
\end{theorem}
\begin{proof}
One implication follows from the definition: if $f \simeq \ast$, then $f|_{BS} \simeq \ast$, and therefore $\ker(f) = S$. Now, assume $\ker(f) = S$, that is,   $f\vert_{BS}\simeq \ast$. We will use the decomposition of the classifying space with respect to the orbit category, see \cref{p:L = hocolim}, to study the mapping space. We have the following equivalences
\[
\map(\Linp,Z)_{\{c\}}\simeq \map(\hocolim_{\OFc} \tilde BP,Z)_{\{c\}} \simeq \holim{\OFc}\map(\tilde BP,Z)_{c}\simeq \holim{\OFc} \underline Z
\]
where $\map(\Linp,Z)_{\{c\}}=\{f\colon \Linp\to Z\mid f\vert_{BS}\simeq \ast \}\subset \map(\Linp,Z)$, and $\underline Z$ denotes the constant functor with value $Z$. The first equivalence holds because $Z$ is $p$-complete and the last equivalence follows from the fact that $Z$ is $\Sigma BP$-null.
We also have the following equivalences:
\[
\holim{\OFc} \underline Z\simeq \map (\hocolim_{\OFc} \underline{*}, Z) \simeq \map (|\OFc|,Z).
\]
Since $Z$ is $p$-complete, \cref{p:|O|=punt} finally gives
\[
\map(\Linp,Z)_{\{c\}}\simeq \map(|\OFc|,Z)\simeq \map(|\OFc|^\wedge_p,Z)\simeq Z.
\]
In particular, $\map(\Linp,Z)_{\{c\}}$ is connected and equal to $\map(\Linp,Z)_{c}$, that is, every map $f\in  \map(\Linp,Z)_{\{c\}}$ is null-homotopic.
\end{proof}
\begin{remark}\label{r:propertykernel}
Let $(S,\F)$ be a saturated fusion system, $Z$ be a connected $p$-complete $\sbz{p}$-null space, and $f \colon \Linp \rightarrow Z$. Let $\alpha \colon P\rightarrow S$ be a group homomorphism, then $f\circ \Theta \circ B\alpha$ is null-homotopic if and only if $\alpha(P)\leq \ker(f)$. In this situation, $\Theta \circ B\alpha$ factors though the homotopy fiber of $f$.
\end{remark}

In the spirit that the kernel of a group homomorphism is a normal subgroup, the kernel of a map satisfies a weaker closure property.

\begin{defi}
Let $(\F,S)$ be a saturated fusion system and $Q\le S$. $Q$ is said to be 
\emph{strongly $\F$-closed} if for all $P \leq Q$ and all morphism  $\varphi\in \Hom_\F(P,S)$ we have $\varphi(P) \leq Q$.
\end{defi}
\begin{remark}
Notice that if $Q\le S$ is strongly $\F$-closed in a saturated fusion system $(\F,S)$, then it follows from the definition that $Q$ is normal subgroup of $S$ since conjugation by elements of $S$ are morphisms of $\F$.
\end{remark}

\begin{example}
If $G$ is a finite group and $S \in \Syl_p(G)$,  $Q \leq S$ is strongly $\F_S(G)$-closed if and only if $Q$ is \textit{strongly closed in $G$}, i.e., if for all $q \in Q$ and 
$g \in G$ such that $c_g(q) \in S$, then $c_g(q) \in Q$. Note that if $G=S$ is a finite $p$ group, then $Q\lhd S$.
\end{example}

\begin{theorem}
\label{t:characterization_kernels}
Let $(S,\F)$ be a saturated fusion system and $K\leq S$. Then $K$ is the kernel of a map $f \colon \Linp \rightarrow Z$ where $Z$ is a $p$-complete and $\Sigma B\mathbb{Z}/p$-null space if and only if $K$ is a strongly $\F$-closed subgroup. Moreover, if $K$ is strongly $\F$-closed, then there exists a fusion system $\F'$ and a map $f'\colon \Linp\to \Linp'$ such that $K=\ker(f')$.
\end{theorem}
\begin{proof}
First we assume that $K=\ker(f)$ for $f\colon \Linp \rightarrow Z$ where $Z$ is a $p$-complete and $\Sigma B\mathbb{Z}/p$-null space. Let $P \leq \ker(f)$ and $\varphi \in \Hom_{\F}(P, S)$. The following homotopy commutative diagram
\[
\begin{tikzcd}[row sep= small]
BP \arrow[dd, "B\varphi"', "\simeq"] \arrow[rd, "B\delta_P"] &               &   \\
& B\F \arrow[r,"f"] & Z \\
B\varphi(P) \arrow[ru, "B\delta_{\varphi(P)}"']    &               &  
\end{tikzcd}
\]
shows that $f\circ B\delta_{P}$ is null-homotopic iff $f\circ B\delta_{\varphi(P)}$ is null-homotopic. Therefore $\varphi(P)\leq \ker(f)$.

Conversely, let $K$ be a strongly $\mathcal{F}$-closed subgroup of $S$. Consider the group homomorphism $\rho\colon S\to S/K\to \Sigma_{|S/K|}$ which is the composite of the projection followed by the regular representation of $S/K$. Recall that $\rho$ is a fusion preserving morphism if, for any isomorphism $\varphi\colon P\to \varphi(P)$ in $\F$, the actions of $P$ via ${\iota\colon P\hookrightarrow S\to S/K}$ and ${\phi \colon P\overset\varphi \to \varphi(P)\hookrightarrow S\to S/K}$ define isomorphic $P$-sets $(S/K,\iota)$ and $(S/K,\phi)$
. According to \cite[Theorem B]{MR2518640}, if $\rho$ is a fusion invariant morphism, then there is a non-negative integer $m \geq 0$ and a map 
\[
f\colon \Linp\to \p{B(\Sigma_{|S/K|}\wr \Sigma_{p^m})}{p},
\]
where $\wr$ denotes the wreath product, such that $f\circ B\delta_{P}$ is homotopic to the composite 
\[
BS \overset{\p{(\Delta B \rho)}{p}}\longrightarrow  \p{(B(\Sigma_{|S/K|})^{p^m})}{p} \overset{\p{\Delta}{p}}\longrightarrow  \p{B(\Sigma_{|S/K|} \wr \Sigma_{p^m})}{p}.
\] 
In particular $\ker(f)=K$. The proof is then reduced to show that $\rho$ is fusion invariant. 

First, we can describe $(S/K,\phi)$ as

\[
(S/K, \phi) \cong \Iso^{\ast}(\varphi) \Res_{\varphi(P)}^S (S/K) \cong \Iso^{\ast}(\varphi) \Res_{\varphi(P)}^S\Ind_K^S(\ast).
\]
Applying the Mackey formula to $\Res_{\varphi(P)}^S\Ind_K^S$, we get
\begin{equation}\begin{split}\label{equation:S/K P set}
(S/K, \phi) &\cong \coprod_{[x] \in \varphi(P) \backslash S/K} \Iso^{\ast}(\varphi) \Ind_{\varphi(P) \cap K^x}^{\varphi(P)} 
\Iso^{\ast}(c_x) \Res_{(\varphi(P) \cap K)^x}^K (\ast) \\ 
&=     \coprod_{[x] \in \varphi(P) \backslash S/K} \Ind_{\varphi^{-1}(\varphi(P) \cap K)}^P \Iso^{\ast}(\varphi) 
\Iso^{\ast}(c_x) \Res_{(\varphi(P) \cap K)^x}^K (\ast).
\end{split}\end{equation}
where the second equality comes from the commutativity of isogation and induction and the strongly $\F$-closed condition for $K$. Again, since $K$ is strongly $\F$-closed and $\varphi$ an isomorphism in $\F$ we have $\varphi^{-1}(\varphi(P) \cap K) = P \cap K$.
Therefore, since $\Iso^{\ast}(\varphi) \Iso^{\ast}(c_x) \Res_{(\varphi(P) \cap K)^x}^K (\ast) = \ast$ as $(P \cap K)$-set,  \cref{equation:S/K P set} becomes
\[
(S/K, \phi) \cong \coprod_{[x] \in \varphi(P) \backslash S/K} \Ind_{P \cap K}^P (\ast) \cong \coprod_{l_{\varphi}} P/P \cap K,
\]
where $l_{\varphi} = |\varphi(P) \backslash S/K| = |S/\varphi(P)\cdot K| = |S|/|\varphi(P)\cdot K|$ since $K \lhd S$. Similarly, 
\[
(S/K, \iota) \cong \coprod_{[x] \in P \backslash S/K} \Ind_{P \cap K}^P (\ast) \cong \coprod_{l} P/P \cap K,
\]
where $l = |S|/|P\cdot K|$. Since both $(S/K, \phi)$ and $(S/K, \iota)$ are described as a disjoint union of isomorphic $P$-sets, we conclude that $l=l_\varphi$ and they are isomorphic.
\end{proof}

\begin{example}
Let $f\colon \Linp \rightarrow BP$ be a map between the classifying space of a saturated fusion system $\F$ and the classifying space of a finite $p$ group $P$. The map $f$ satisfies the hypothesis required in \cref{d:kerf} to consider the kernel of $f$ since $\Omega BP$ is homotopically discrete and $BP$ is $p$-complete. Then $\ker(f)=\ker(g)$ where $g\colon S\rightarrow P$ is a group homomorphism such that $Bg\simeq f\circ \Theta$ (which is well-defined by \cite[Theorem 4.4]{MR1992826}). In particular, this situation applies to the map $f\colon \Linp \rightarrow B(\pi_1(B\F))$. From the description of the fundamental group in \cref{prop:hmtpy gps} we obtain $\ker(f)=\OFp$. In particular, we recover the fact that it is a strongly $\F$-closed subgroup of $S$.
\end{example}

We finish this section approaching the following question. Given a saturated fusion system $(S,\F)$ and a map $f \colon \Linp \rightarrow Z$ where $Z$ is a connected $\Sigma B\mathbb Z/p$-null $p$-complete space. Does $f$ 
factors, up to homotopy, through $\tilde{f} \colon B\F' \rightarrow Z$ with trivial kernel, where $\F'$ is a saturated fusion system? Under some hypothesis we can give a positive answer to the previous question, which is related to the construction of quotients in the context of fusion systems and their classifying spaces.

\begin{defi}
Let $(\F,S)$ be a saturated fusion system and $Q\le S$. $Q$ is said to be 
\emph{$\F$-normal} or \emph{normal in $\F$}, denoted $Q\unlhd \F$, if for all $P,R\le S$ and all $\varphi\in \Hom_\F(P,R)$, $\varphi$ extends to a morphism $\overline{\varphi}\in\Hom_\F(PQ,RQ)$ such that $\overline\varphi(Q)=Q$.
\end{defi}
It is straightforward that the following sequences of implications hold for $Q\le S$ (but they are not equivalences, see \cite[Proposition 4.5]{MR2848834} for the first implication)
\begin{center}
$Q$ is normal in $\F$ $\Rightarrow$ $Q$ is strongly $\F$-closed $\Rightarrow$ $Q$ is normal in $S$.
\end{center}
Given an $\F$-normal subgroup $K\leq S$, Oliver and Ventura \cite{MR2317758} studied quotients and extensions of fusions systems and their classifying spaces. They show that there is a fusion system $(\F/K, S/K)$ with a map $f\colon B\F \rightarrow B(\F/K)$ which factors through $BS\rightarrow B(S/K)$ when restricted to Sylow subgroups. 

\begin{prop}
\label{p:K normal}
Let $(S,\F)$ be a saturated fusion system, $Z$ be a $\Sigma B\mathbb Z/p$-null $p$-complete space and $K\leq S$ an $\F$-normal subgroup. There is a map $p\colon \Linp \rightarrow B(\F/K)$ with kernel $K$ inducing an equivalence $$p_*\colon \map(B(\F/K),Z)\stackrel{\simeq}{\rightarrow} \map(B\F,Z)_{\{c\}}$$ where $\{c\}$ denotes the subset of components of those maps nullhomotopic when restricted to $BK$ via $BS\rightarrow B\F$. Moreover, if $p_*(\tilde{f})\simeq f$ then $\ker(\tilde{f})=\ker(f)/K$. 
\end{prop}

\begin{proof}
The starting point is \cite[Section 2]{MR2317758} where the authors describe extensions of fusion systems. In particular, they show that there exists a saturated fusion system $(S/K, \F/K)$ with linking system $\Lin/K$, and a map $pr \colon |\Lin| \rightarrow |\Lin/K|$ which fits in a homotopy fiber sequence $BK \rightarrow |\Lin| \rightarrow |\Lin/K|$ which restricts via $\Theta$ to $BK \rightarrow  BS \rightarrow B(S/K)$ induced by the $\pi \colon S\twoheadrightarrow S/K$.

Since $BK$ is $B\z/p$-acyclic, and $Z$ is $\Sigma B\z/p$-null, as an application of Zabrodsky lemma (see 
\cite[Proposition 3.5]{MR1397728}), we obtain an equivalence 
$$p_*\colon \map(|\Lin/K|,Z)\stackrel{\simeq}{\rightarrow} \map(|\Lin|,Z)_{\{c\}}$$ where $\{c\}$ denotes the subset of components of those maps nullhomotopic when restricted to $BK$ via $BS\rightarrow |\Lin|$. Since $Z$ is $p$-complete we obtain the equivalence
$$p_*\colon \map(B(\F/K),Z)\stackrel{\simeq}{\rightarrow} \map(B\F,Z)_{\{c\}}.$$
It remains to show that if $p_*(\tilde{f})\simeq f$ then $\ker(\tilde{f})=\ker(f)/K$. Note that the equivalence applied to the classifying space of $S$ and $\pi \colon S\twoheadrightarrow S/K$, shows that for any $H\leq S/K$ we have an equivalence $$\map(BH,Z)\stackrel{\simeq}{\rightarrow} \map(B\pi^{-1}(H),Z)_{\{c\}}.$$
Then $\pi^{-1}(H)\leq \ker(f)$ if and only if $H \leq \ker(\tilde{f})$ (note that taking $Z=\Linp/K$ together $f\colon \Linp \rightarrow Z$ gives the commutativity of the required diagrams).
\end{proof}
\begin{coro}
Let $(S,\F)$ be a saturated fusion system, $Z$ be a $\Sigma B\mathbb Z/p$-null $p$-complete space, and  $f\colon \Linp \rightarrow Z$ be a map. If ${\ker(f) \unlhd \F}$ then there exists a map ${\tilde{f}\colon B(\F/\ker(f))\rightarrow Z}$ with trivial kernel which factors $f$.
\end{coro}

There are generic situations in which we can apply \cref{p:K normal}. 

\begin{defi}
We say that a $p$-group $S$ is resistant if it is normal in each saturated fusion system over $S$.
\end{defi}

\begin{coro}
\label{c:normal resistant}
Let $(S,\F)$ be a saturated fusion system, $Z$ be a $\Sigma B\mathbb Z/p$-null $p$-complete space and $f \colon \Linp \rightarrow Z$ be a pointed map with $K\leq \ker(f)$.
If $K\leq S$ is an abelian strongly $\F$-closed subgroup or $S$ is resistant, then there exist a saturated fusion system $(S/K,\F/K)$ and a homotopy factorization $\tilde{f} \colon \Linp/K \rightarrow Z$ with $\ker(\tilde{f})=\ker(f)/K$.
\end{coro}

\begin{proof}
Recall that $\ker(f)$ is strongly $\F$-closed by \cref{t:characterization_kernels}. If $K$ is abelian and strongly $\F$-closed, then it is normal in $\F$ by 
\cite[Corollary 4.7]{MR2848834}. In the other case, if $S$ is resistant, each strongly $\F$-closed subgroup is also normal in $\F$ (see \cite[Proposition 4.5]{MR2848834}). Finally, apply \cref{p:K normal}.
\end{proof}

\section{Cellular properties of the classifying space of a saturated fusion system}\label{s:cw BF}

The goal of this section is to prove the main theorem of the paper. Given a finite $p$-group $P$, this result characterizes the property of being $BP$-cellular for classifying spaces of saturated fusion systems in terms of the fusion data.

\begin{defi}
Let $(\F,S)$ be a saturated fusion system. We define $Cl_{\F}(P)$ to be the smallest strongly $\F$-closed subgroup of $S$ that contains $\varphi(P)$ for all $\varphi \in \Hom(P,S)$.
\end{defi}
Notice that given a finite $p$-group $P$, $Cl_{\F}(P)$ is well-defined since the intersection of strongly $\F$-closed subgroups is again strongly $\F$-closed.

\begin{remark}\label{r:N<CL<NO}
Let $N$ be the normal subgroup of $S$ generated by $\varphi(P)$, for all $\varphi\in \Hom(P,S)$. Since $\OFp$ is strongly $\F$-closed, $N\OFp$ is so by \cite[Proposition A.9]{MR2835340}, and we have inclusions $N\leq Cl_{\F}(P)\leq N\OFp \leq S$.
\end{remark}

\begin{theorem}
\label{t:BP << Lp}
Let $(S,\F)$ be a saturated fusion system and let $P$ be a finite $p$-group. Then $\Linp$ is $BP$-cellular if and only if $S = Cl_{\F}(P)$.
\end{theorem} 

We delay the proof of the theorem to describe some straightforward consequences.

\begin{coro}
\label{c:BS << Lp}
Let $(S,\F)$ be a saturated fusion system and let $P$ be a finite $p$-group.
\begin{enumerate}[(a)]
\item The classifying space $\Linp$ is $BS$-cellular.
\item Let $(S,\F')$ be a saturated fusion system with $\F\subset \F'$. If $B\F$ is $BP$-cellular then $B\F'$ is also $BP$-cellular. In particular, if $BS$ is $BP$-cellular then so is $\Linp$.
\item Let $Q\twoheadrightarrow P$ be an epimorphism of finite $p$-groups. If $\Linp$ is $BP$-cellular, then it is also $BQ$-cellular. 
\item Let $A$ be a pointed connected space. If $Cl_\F((\pi_1 A)_{ab})=S$, then $\Linp$ is $A$-cellular.
\item Let $\Omega_{p^m}(S)$ be the (normal) subgroup of $S$ generated by its elements of order $p^i$, with $i \leq m$. Then $\Linp$ is $\bz{p^m}$-cellular if and only if 
$S = Cl_{\F}(\Omega_{p^m}(S))$. In particular, there is a non-negative integer $m_0 \geq 0$ such that $\Linp$ is $\bz{p^m}$-cellular for all $m \geq m_0$.
\end{enumerate}
\end{coro}

\begin{proof}
\begin{enumerate}[(a)]
\item Since $Cl_{\F}(S) = S$, it is a direct application of \cref{t:BP << Lp}.
\item It follows from the inclusions $Cl_\F(P)\leq Cl_{\F'}(P)\leq S$. Note that $\F_S(S)\subset \F$.
\item It is clear from the inclusion $Cl_\F(Q)\leq Cl_\F(P)\leq S$ (it could also be deduced from the fact that $BP$ is $BQ$-cellular).
\item Since $SP^\infty A \simeq \prod_{i \geq 1} K(H_i(A;\z{}),i)$ is $A$-cellular by \cite[Corollary 4.A.2.1]{MR1392221}, then $B(\pi_1 A)_{ab} \simeq K(H_1(A;\z{}),1)$ is 
also $A$-cellular from \cite[2.D]{MR1392221}. Finally, $B\F$ is $B(\pi_1 A)_{ab}$-cellular by (a) if $Cl_\F((\pi_1 A)_{ab})=S$.
\item From the definition we have an equality $Cl_\F(\Omega_{p^m}(S))\cong Cl_\F(\mathbb Z/p^m)$. Now, $B\Omega_{p^m}(S)$ is $B\mathbb Z/p^m$-cellular and there exists an $m_0 \geq 0$ such that $S$ is generated by elements of order a power of $p$ less than or equal to $p^{m_0}$.\qedhere
\end{enumerate}
\end{proof}

\begin{remark}
From \cref{c:BS << Lp}, we see that $BP$ is $BQ$-cellular if there is an epimorphism $Q\twoheadrightarrow P$. In particular $\bz{p}$ is $BQ$-cellular for any finite $p$-group $Q$. This fact was already known due to \cite[Corollary 7.2]{MR3416113}.
\end{remark}
\begin{remark}
Given a finite group $G$, it follows from \cref{c:BS << Lp} that $BG^\wedge_p$ is $BS$-cellular. In particular, it is $BS$-acyclic which implies that it is $\bz{p}$-acyclic (by \cite[Lemma 6.13]{MR1397728}). This result was previously established by Flores in \cite[Proposition 3.14]{MR2272149}.
\end{remark}
The strategy of proof for \cref{t:BP << Lp} goes by analyzing  the fiber sequence described by Chachólski in \cite[Theorem 20.5]{MR1408539}
\[
CW_{BP}(\Linp)\overset{c}{\longrightarrow}\Linp \overset{r}{\longrightarrow}P_{\Sigma BP} C,
\]
where $C$ is the homotopy cofiber of the evaluation map $ev \colon \bigvee_{[BP, \Linp]_{\ast}} BP \rightarrow \Linp$ and $r$ is the composite 
$\Linp \rightarrow C \rightarrow P_{\Sigma BP} C$. Since $\pi_1(\Linp)$ is a finite $p$-group, by \cite[Proposition~2.9]{MR1257059}, the same holds for $P_{\Sigma BP} C$. Therefore, Bousfield-Kan $p$-completion of the previous homotopy fibration is again a homotopy fiber sequence (\cite[II.5.1]{MR0365573})
\[
CW_{BP}(\Linp)^\wedge_p\overset{c^\wedge_p}{\longrightarrow} \Linp \overset{r^\wedge_p}\longrightarrow \p{(P_{\Sigma BP}C)}{p}. 
\] 

The strategy is  to compute the kernel of $r^\wedge_p$. 
To apply the theory of kernels, developed in \cref{s:ker f}, we need to show that $\p{(P_{\Sigma BP}C)}{p}$ is a $p$-complete $\sbz{p}$-null space.

\begin{lemma}
\label{l:Sigma null}
If $X$ is a simply connected space and $P$ is a finite $p$-group, then $\p{(P_{\Sigma BP}X)}{p}$ is $\Sigma B\mathbb Z/p$-null.
\end{lemma}

\begin{proof} We have the following weak homotopy equivalences
$$\map_*(\Sigma B\mathbb Z/p, \p{(P_{\Sigma BP}X)}{p})\simeq \map_*( B\mathbb Z/p, \Omega \p{(P_{\Sigma BP}X)}{p})\simeq \map_*( B\mathbb Z/p, \p{(\Omega P_{\Sigma BP}X)}{p})$$
where the last equivalence holds by \cite[V.4.6 (ii)]{MR0365573} since $X$ is simply connected (and so is $P_{\Sigma BP}X$ by \cite[2.9]{MR1257059}). 

The commutation rules between nullification functors and loops in \cite[3.A.1]{MR1392221} show that $\Omega P_{\Sigma BP}X\simeq P_{BP}(\Omega X)$. Finally,
$\map_*( B\mathbb Z/p, \p{(P_{BP}\Omega X)}{p})\simeq \map_*( B\mathbb Z/p, (P_{BP}\Omega X))\simeq *$ where the first equivalence follows from Miller's theorem \cite[Thm 1.5]{MR750716} and the second from the fact that $BP$ is $B\mathbb Z/p$-acyclic (\cite[Lemma 6.13]{MR1397728}).
\end{proof}

A key step in the proof of \cref{t:BP << Lp} is the following computation of the kernel of the map $r^\wedge_p$.

\begin{prop}
\label{p:Kerf = Cl}
Let $(S,\F)$ be a saturated fusion system. Then $\ker(\p{r}{p}) = Cl_{\F}(P)$.
\end{prop}

\begin{proof}
We start by showing that $Cl_{\F}(P)\leq \ker(\p{r}{p})$. Since $\ker(\p{r}{p})$ is strongly $\F$-closed, it is enough to check that $f(P)\leq \ker(\p{r}{p})$ for any morphism $f\colon P\rightarrow S$. Let $Bf\colon BP\rightarrow BS$ be the induced map between classifying spaces. Given the inclusion of the Sylow $p$-subgroup $\Theta$, the composite $\Theta\circ Bf\colon BP\rightarrow BS \rightarrow B\F$ belongs to the mapping space $\mapp(BP,\Linp)$. But, by the universal property of  $CW_{BP}$, 
\[\mapp(BP,\Linp)\simeq \mapp(BP,CW_{BP}(\Linp)).\]
so the composite $r\circ \Theta \circ Bf$ is null homotopic. In particular, $\p{r}{p}\circ \Theta \circ Bf$ is also null-homotopic, and  $f(P)\leq \ker(\p{r}{p})$.

Next we check that $\ker(\p{r}{p}) \leq Cl_{\F}(P)$. According to \cref{t:characterization_kernels}, there exists a map $k \colon \Linp \rightarrow Z$ such that $Z$ is $p$-complete and $\Sigma B\z/p$-null such that
$\ker(k) = Cl_{\F}(P)$.

Let $\iota \colon B\ker(\p{r}{p})\rightarrow BS\overset\Theta \rightarrow B\mathcal F$ be the inclusion via the Sylow $p$-subgroup of the $\ker(\p{r}{p})$.
Then, by \cref{r:propertykernel}, $\ker(\p{r}{p})\leq \ker(k)=Cl_{\F}(P)$ if the composite $k\circ \iota$ is null-homotopic. By definition of the kernel, $\p rp\circ i\simeq \ast$, so there is $\tilde{\iota}\colon B\ker(\p{r}{p})\rightarrow CW_{BP}(\Linp)$ such that $c\circ \tilde{\iota}\simeq \iota$.
\[
\begin{tikzcd}
& \p{CW_{BP}(\Linp)}{p} \arrow[rd, "\p{(k\circ c)}p"] \arrow[d, "\p{c}{p}"] &   \\
B\ker(\p{r}{p}) \arrow[ru, "\tilde \iota", dashed] \arrow[r, "\iota"] \arrow[rd, "\simeq\ast"'] & \Linp \arrow[d, "\p{r}{p}"] \arrow[r, "k"]                                & Z \\
& \p{(P_{\Sigma BP} C)}{p}.                                                 & 
\end{tikzcd}
\]
But then we can reduce further to show that $k\circ c\colon CW_{BP}(\Linp)\rightarrow Z$ is null-homotopic. Notice that $Z$ is $\sbz{p}$-null then, by \cref{lem:Bzpnull->BPnull}, $Z$ is $\Sigma BP$-null. Then, by \cref{p:mapCWtoZ}, $k\circ c$ is null-homotopic if and only if $(k\circ c)_*\colon [BP,{CW_{BP}(\Linp)}]_* \rightarrow [BP,Z]_*$ is trivial. That is, for any pointed map $f\colon BP\to CW_{BP}(\Linp)$, the composite $k\circ c\circ f\simeq \ast$. By \cite[Theorem 4.4 (a)]{MR1992826}, $c\circ f\colon BP\to B\F$ is homotopic to $\Theta\circ B\rho\colon BP\to BS\to \Linp$, where $\rho \in \Hom(P,S)$. Since we have defined $k$ with the property $\ker(k)=Cl_\F(P)$, then 
$
k\circ c\circ f\simeq k\circ \Theta\circ B\rho\simeq \ast.
$
\end{proof}

We are now ready to prove the main theorem of this paper.

\begin{proof}[Proof of \cref{t:BP << Lp}]
First assume that $\Linp$ is $BP$-cellular. Then $P_{\Sigma BP} C$ is contractible and ${r \simeq *}$. This implies that
$
Cl_\F(P)\cong \ker(\p r p)=S
$
where the first equality holds by \cref{p:Kerf = Cl} and the second one by the contractibility of $r$.

Now assume that $S=Cl_\F(P)$ and consider the Chachólski fibration
\[
CW_{BP}(\Linp)\overset{c}\longrightarrow\Linp \overset{r}\longrightarrow P_{\Sigma BP}C.
\]
We first prove that $P_{\Sigma BP}C$ is simply connected and $\Sigma B\mathbb Z/p$-null. Let $N$ be the normal subgroup of $S$ generated by the images of all group morphism $\varphi\colon P\to S$. By \cref{r:N<CL<NO} we have
\begin{equation}\label{eq:prf:CL<NO<S}
Cl_{\F}(P)\leq N\OFp \leq S
\end{equation}
But $S=Cl_\F(P)$, so the inequalities of \eqref{eq:prf:CL<NO<S} become equalities. Then \cref{r:Pi1ofC} implies that  $C$ is simply connected, and therefore, $P_{\Sigma BP}C$ is so and \cref{l:Sigma null} implies  that  it is also $\Sigma B\mathbb Z/p$-null.
Notice that $P_{\Sigma BP}C$ being simply connected not only ensures that we can apply the theory of kernels but also implies that it is nilpotent. Moreover,
\[
\tilde{H}_*\left(\Linp;\z\left[\tfrac{1}{p}\right]\right)\hookrightarrow  \tilde{H}_*\left(BS;\z\left[\tfrac{1}{p}\right]\right)=0.
\]
Then, by Miller's theorem \cite[Theorem 1.5]{MR750716}, $r$ is null-homotopic if and only if $\p{r}{p}$ is null-homotopic, and by \cref{t:Dwyer p local}, if and only if $\ker(\p r p)=S$. Then we are done by \cref{c:ifrnull_CW}.
\end{proof}

Let $G$ be a finite group. The case when $G$ is generated by elements of order $p$ is well studied by R. Flores and R. Foote in \cite{MR2823972}. We start by giving a simple example where $G$ is not generated by elements of order $p^i$.

\begin{example}
Consider the permutation group on $3$ elements, $\Sigma_3$. It is generated by transpositions, i.e, by elements of order $2$, but the Sylow $3$-subgroup of $\Sigma_3$ is $S=\z{/3}$. Therefore, $BS$ is $\bz{3^r}$-cellular for all $r \geq 1$ and hence, by \cref{c:BS << Lp}, $\p{(B \Sigma_3)}{3}$ is so.

But, notice that $B\Sigma_3$ is not $\bz{3^r}$-cellular for any $r \geq 1$: Applying $\map_*(\bz{3^r},-)$ to the homotopy fiber sequence
\[
\bz{3} \overset{Bi}{\rightarrow} B\Sigma_3 \rightarrow \bz{2},
\]
we see that $Bi$ is a $\bz{3^r}$-equivalence for any $r$, since $\map_*(\bz{3^r},\bz 2)\simeq *$. Then,  
\[CW_{\bz{3^r}}(B\Sigma_3) \simeq CW_{\bz{3^r}}(\bz{3}) \simeq \bz{3}.\]
\end{example}
We end this section by describing some situations in which we can describe $CW_{BP}(\Linp)$ when $Cl_{\F}(P) \neq S $. The strategy is to identify $\p{P_{BP}(C)}{p}$ in the $p$-completed Chachólski fibration describing $CW_{BP}(B\F)$.
\begin{coro}
\label{c:normal}
Let $(S,\F)$ be a fusion system and let $P$ be a finite $p$-group. If $Cl_{\F}(P) \vartriangleleft \F$, then $CW_{BP} (\Linp)$ is homotopy equivalent to the homotopy
fiber of $\Linp \rightarrow B(\F/Cl_{\F}(P))$.
\end{coro}

\begin{proof}
Let $K = Cl_{\F}(P)$. Since $K$ is normal in $\F$, there is a saturated fusion system $(S/K, \F/K)$, and a map defined between the nerve of the associated linking systems $p \colon |\Lin| \rightarrow |\Lin/K|$ whose homotopy fiber is $BK$. By \cref{p:Kerf = Cl}, $Cl_\F(P)=\ker(\p{r}{p})$, so \cref{p:K normal} gives us an injective factorization of $\p{r}{p}$, $\tilde{r}\colon B(\F/Cl_\F(P))\rightarrow \p{P_{\Sigma BP}(C))}{p}$, unique up to homotopy. 

Now consider $\p{p}{p}\colon \Linp\rightarrow B(\F/Cl_\F(P))$,  the restriction $CW_{BP}(\Linp)\rightarrow B(\F/Cl_\F(P))$ is a null homotopic map since we are in the hypothesis of \cref{p:mapCWtoZ} and 
\[
{[BP,CW_{BP}(\Linp)]\rightarrow [BP, B(\F/Cl_\F(P))]}
\]
is trivial  (any map $BP\rightarrow \Linp$ factors through $Cl_\F(P)$).
Moreover, in this situation, $\map(CW_{BP}(\Linp),P_{\Sigma BP}(C))_c\rightarrow P_{\Sigma BP}(C)$ is an equivalence. The hypotheses of Zabrodsky lemma are satisfied and we have a factorization of $\p{p}{p}$, via $\tilde{p}\colon P_{\Sigma BP}(C)\rightarrow B(\F/Cl_\F(P))$. 

Note that the composites $\tilde{r}\circ \p{\tilde{p}}{p}$ and $\p{\tilde{p}}{p}\circ \tilde{r}$ are homotopic to the identity by the uniqueness of such factorizations since they factor the identity on $\Linp$.

Finally, we get that the homotopy fiber of $\p{r}{p}\colon \Linp \rightarrow B(\F/Cl_\F(P))$ is $\p{CW_{BP}(\Linp)}{p}$ whose $\Q$-completion is contractible since the $\Q$-completion of the classifying space of a fusion system is so. Then this homotopy fiber is in fact $CW_{BP}(\Linp)$.
\end{proof}

\begin{example}\label{e:NG(S) controls fus}
Let $G$ be a finite group and let $S$ be a $p$-Sylow subgroup of $G$. Assume that $N_G(S)$ controls fusion in $G$. Then $\p{BN_G(S)}{p} \simeq \p{BG}{p}$ and $S$ is normal in 
$\F_S(N_G(S))$. On account of \cref{c:normal}, for all finite $p$-group $P$, $CW_{BP} (\p{BG}{p})$ is equivalent to the homotopy fiber of 
$\p{BN_G(S)}{p} \rightarrow \p{B(N_G(S)/Cl_{\F_S(N_G(S))}(P))}{p}$.
\end{example}

\begin{example}
Let $G = \z{/p^n} \wr \z{/q} = (\z{/p^n})^q \rtimes \z{/q}$, when $p \neq q$ and $n \geq 1$. The Sylow 
$p$-subgroup of $G$ is $S = (\z{/p^n})^q$ and $Cl_{\F_S(G)}(\z{/p^r}) = (\z{/p^r})^q$ is abelian and hence, by \cref{c:normal resistant}, normal in $\F_S(G)$. Applying  \cref{t:BP << Lp},
$\p{BG}{p}$ is $\bz{p^r}$-cellular if and only if $r \geq n$. Then $CW_{\bz{p^r}}(\p{BG}{p})$ is equivalent to the homotopy fiber of 
$\p{BG}{p} \rightarrow \p{B(G/(\z{/p^r})^q)}{p}$ by \cref{c:normal}.
\end{example}

\begin{example}
Other explicit examples appear in \cite[Example 5.2]{MR2351607}. The authors proved that the normalizer of the Sylow of the Suzuki group
$\sz(2^n)$, with $n$ an odd integer at least $3$, is $N_{\sz(2^n)}(S) = S \rtimes \z{/(2^n-1)}$ and it controls fusion in $\sz(2^n)$. In this case, $S$ is $\bz{2^m}$-cellular for
all $m \geq 2$ and hence $\p{\bsz(2^n)}{2}$ is so. Moreover, $Cl_{\F_S(\sz(2^n))}(\z{/2}) \cong (\z{/2})^n$ and hence $CW_{\bz{2}}(\p{\bsz(2^n)}{2})$ is equivalent to the 
homotopy fiber of 
\[
\p{B N_{\sz(2^n)}(S)}{2} \rightarrow \p{B(N_{\sz(2^n)}(S)/(\z{/2})^n)}{2}.
\]
\end{example}

\section{The cellularization of the classifying spaces of a family of exotic fusion systems}%
\label{s:exotic example} 

In this section we show how our methods apply to describe the cellularization of classifying spaces of a family of exotic fusion systems. The main reference for this section is \cite{MR2272147}.

Let $S=B(3,r;0,\gamma,0)$, for $r \geq 4$ and $\gamma = 0,1,2$, be the family of finite $3$-groups of order $3^r$ generated by $\{s,s_1, \ldots, s_{r-1}\}$ and relations
\begin{itemize}\setlength{\itemsep}{.5em}
\item $s_i = [s_{i-1},s]$ for all $i \in \{2, \ldots, r-1\}$,
\item $[s_1, s_i] = 1$ for all $i \in \{2, \ldots, r-1\}$,
\item $s_1^3 s_2^3 s_3^{\phantom{3}} = s_{r-1}^\gamma$,
\item $s_i^3 s_{i+1}^3 s_{i+2}^{\phantom{3}} = 1$ for all $i \in \{2, \ldots, r-1\}$, and assuming $s_r=s_{r+1}=1$.
\item $s^3=1$.
\end{itemize}
In \cite[Proposition A.9]{MR2272147}, the authors show that the center of $S$ is the subgroup
$
{Z(S)=\langle s_{r-1}\rangle.}
$
The normal subgroup $\langle s_1,\ldots,s_{r-1} \rangle=\langle s_1,s_2\rangle$ of $S$ is of index $3$ and the corresponding group extension is split. There are group isomorphisms
\begin{equation}\label{equation:exotic semidirect product}
B(3,r;0,\gamma,0) = \langle s_1,s_2 \rangle \rtimes \langle s \rangle = \left\{ \begin{array}{ll}
(\z{/3^m} \times \z{/3^m})     \rtimes \z{/3} & \mbox{, if $r = 2m + 1$,} \\
(\z{/3^m} \times \z{/3^{m+1}}) \rtimes \z{/3} & \mbox{, if $r = 2m$.}
\end{array}
\right.
\end{equation}
In \cite[Theorem 5.10]{MR2272147}, the authors construct families of exotic $3$-local finite groups $\F$ whose Sylow $3$-subgroup is $S=B(3,r;0,\gamma,0)$. 

\begin{prop}\label{c:exotic example}
Let $\F$ be an exotic fusion system over $B(3,r;0,\gamma,0)$ such that $\F$ contains  a rank two elementary abelian group which is $\F$-Alperin (see the classification given in  \cite[Theorem 5.10]{MR2272147}).
Then 
\begin{enumerate}[(i)]
\item If $\gamma = 0$, then $\Linp$ is $\bz{3^l}$-cellular for all $l \geq 1$.
\item If $\gamma \neq 0$, then $\Linp$ is $\bz{3^l}$-cellular if and only if $l \geq 2$. For $l=1$, we have ${Cl_{\F}(\z{/3})=\langle s,s_2\rangle}$.
\end{enumerate}
\end{prop}
\begin{proof}
For short we denote by $S$ the group $B(3,r;0,\gamma,0)$. By \cref{t:BP << Lp} we are reduced to the computation of $Cl_{\F}(\z{/3^l})$. By \eqref{equation:exotic semidirect product}, one deduces that $s\in Cl_\F(\z/3^l)$ for all $l\ge 1$. By \cite[Lemma A.10]{MR2272147}, the subgroup $N=\langle s,s_2\rangle$ of $S$ is a proper normal subgroup of index $[N : S]=3$ and
 either
 \[
Cl_\F(\z/3^l)=S\quad\text{ or }\quad Cl_\F(\z/3^l)=N
 \]
Then, $Cl_\F(\z/3^l)=S$ if and only if there exists $x\in S\setminus N$ such that $x^{3^l}=1$ for all $l\ge 1$. 

If $\gamma = 0$, then the element $ss_1\in S\setminus N$ is of order $3$. This holds since 
\begin{equation}\label{equation:x^3}
(ss_1^js_2^k)^{3^l}=(s_{r-1})^{\gamma j3^{l-1}}\text{ for any choice of }\gamma.
\end{equation}
In particular $(ss_1)^3=(s_{r-1})^{\gamma}=(s_{r-1})^{0}=1$.

If $\gamma \neq 0$, the same element $ss_1$ is of order $9$, $x^9=(s_{r-1})^{\gamma3}=1$, see \eqref{equation:x^3}. Now we show that there is no element in $S\setminus N$ of order $3$. By \eqref{equation:exotic semidirect product}, any element $x\in S$ can be written as $s^i_{\phantom{0}}s_1^js_2^k$ for $i=0,2,1$. 

First, notice that the case $i=2$ can be dropped since the index of $Cl_\F(\z{/3^l})$ in $S$ is either $1$ or $3$. 

Next, we prove that there is no such $x=s^i_{\phantom{0}}s_1^js_2^k$ with $i=0$. Assume that there exists $x=s_1^js_2^k$ of order $3$. Since $x^{3}=1$, and $s_1$ and $s_2$ commute, then $x^{3}=1$ if and only if $s_1^{3^lj}=s_2^{3^lk}=1$.
In particular, $j|3^m$, see \eqref{equation:exotic semidirect product}. Then $s_1^j=(s_1^3)^{j'}=(s_{r-1}^\gamma s_{3}^{-1}s_{2}^{-3})^{j'}\in N$, which is a contradiction.
 
Finally, the only option is $x=ss_1^js_2^l$. Again, by \eqref{equation:x^3}, $x^3=(s_{r-1})^{\gamma j}$. Therefore, $x^{3}=1$ implies $\gamma j \equiv 0 \mod 3$. But $\gamma\neq 0$, so $j\equiv 0 \mod 3$ and this implies that $s_1^j\in N$ which implies that $x\in N$.
\end{proof}

We will finish this section by describing $CW_{B\mathbb Z/3}(B\mathcal F)$ when $\F$ is an exotic fusion system over $B(3,r;0,\gamma,0)$ with $\gamma \neq 0$ such that $\F$ has at least one $\F$-Alperin rank two elementary abelian $3$-subgroup.

\begin{lemma}\label{l:exotic1connected}
Let $\F$ be an exotic fusion system under the hypothesis of \cref{c:exotic example}, $\Linp$ is simply connected.
\end{lemma}

\begin{proof}
By \cref{prop:hmtpy gps}, the fundamental group of $\Linp$ is 
$\pi_1(\Linp) \cong S/\OFp$.
The subgroup $\OFp$ is a strongly $\F$-closed subgroup of $S$, and the arguments in the proof in \cite[page 1751]{MR2272147} show that it must contain $N=\langle s,s_2\rangle < S$. We will show that $\OFp=S$ by proving that $s_1\in \OFp$. Checking tables in  \cite[Theorem 5.10,Lemma A.14]{MR2272147}, we see that the automorphisms of order two $\eta$ and/or $\omega$ are group elements in $\Aut_{\F}(S)$. By the description given there $\eta(s_1)=s_1^{\phantom{f'}}s_2^{f''}$ and $\omega(s_1)=s_1^{-1}s_2^{f''}$. Then $s_1^{-1}\eta(s_1)=s_1^{-2}s_2^{f''}$ or $s_1^{-1}\omega(s_1)=s_1^{-2}s_2^{f''}$ are elements of $\OFp$, since $s^{\phantom{v}}_2,s_1^3\in N\subset \OFp$ then $s_1\in \OFp$. 
\end{proof}

\begin{prop}\label{l:exotic higherlimits}
Let $\F$ be an exotic fusion system satisfying the hypothesis of \cref{c:exotic example} with $\gamma\neq 0$. Let $N=\langle s,s_2\rangle< S$, then there exists a unique map (up to homotopy) $f\colon B\F\rightarrow \p{(B\Sigma_3)}{3}$ whose kernel is $N$.
\end{prop}

\begin{proof}

The proof of \cref{t:characterization_kernels} shows that the quotient morphism $S\rightarrow S/N\cong \mathbb Z/3$ gives a fusion preserving homomorphism $\rho \colon S\rightarrow \Sigma_3$. We want to show that this morphism extends to a map $f:\Linp\rightarrow (B\Sigma_3)^\wedge_3$.

By \cref{p:L = hocolim}, $\Linp \simeq \p{(\hocolim_{\OFc} \tilde{B}P)}{3}$, where $\tilde{B}P \simeq BP$ for $P \in \F^c$. The fusion preserving property of $\rho$ shows that $B\rho \in \lim_{\OF} [BP,\p{(B\Sigma_3)}{3}]$.  

The obstructions for rigidifying the homotopy commutative diagram in the category of spaces lie in $\lim_{\OFc}^{i+1} \pi_{i}(\map(BP, \p{(B\Sigma_3)}{3})_{\Theta'\circ B\rho|_P})$, for $i \geq 1$ (see \cite{MR928836}). Note that since the $3$-Sylow subgroup of $\Sigma_3$ is abelian,  we have $\pi_{1}(\map(BP, \p{(B\Sigma_3)}{3})_{\Theta'\circ B\rho|_P})$ is abelian being a quotient of $C_{\mathbb Z/3}(\rho(P))$. In fact, it will be trivial or $\mathbb Z/3$ (see \cite[Theorem 6.3]{MR1992826}).

We will show that for any $F\colon \OF\rightarrow \mathbb Z_{(p)}-Mod$, $\lim^{i}_{\OF}F=0$ for $i>1$.
From \cite[Proposition 3.2, Corollary 3.3]{MR1992826}, we are reduced to show that derived limits of atomic functors have the same vanishing property. Note that from \cite[Theorem 5.10,Lemma A.14]{MR2272147},  the relevant outer automorphism groups $\Out_\F(P)$ are $SL_2(\mathbb F_3)$ or  $GL_2(\mathbb F_3)$. In both cases the $3$-Sylow subgroup is of order $3$, and then \cite[Proposition 6.2(i)]{MR1154593} implies the result.

The obstructions to uniqueness lie in  $\lim_{\OFc}^{i} \pi_{i}(\map(BP, \p{(B\Sigma_3)}{3})_{\Theta'\circ B\rho|_P})$, for $i \geq 1$ (see \cite{MR928836}). By the previous paragraph we have to look at the first derived functor of atomic functors with value $\mathbb Z/3$. But since $\Aut(\mathbb Z/3)\cong \mathbb Z/2$, by \cite[Proposition 6.1(ii)]{MR1154593} any element of order $3$ will act trivially on $\mathbb Z/3$. Finally note that if a map $g\colon B\F\rightarrow \p{(B\Sigma_3)}{3}$ has kernel $N$, its restriction $g|_{BS}\colon BS\rightarrow \p{(B\Sigma_3)}{3}$ has to be homotopic to $B\rho$.
\end{proof}

\begin{prop}
Let $\F$ be an exotic fusion system over $B(3,r;0,\gamma,0)$ with $\gamma \neq 0$ such that $\F$ has at least one $\F$-Alperin rank two elementary abelian $3$-subgroup given in \cite[Theorem 5.10]{MR2272147}.  Then there exists a map $f\colon B\F \rightarrow (B\Sigma_3)^\wedge_p$ such that $CW_{B\mathbb Z/3}(B\mathcal F)$ is the \mbox{homotopy fiber of $f$.}
\end{prop}

\begin{proof}
Let $f\colon \Linp\to \p{(\Sigma_3)}{3}$ be the map constructed in \cref{l:exotic higherlimits} with $\ker(f)=Cl_\F(S)$. Precisely because of this, $f\circ ev\simeq *$ where $ev\colon \bigvee_{[\bz3,\Linp]_*} B\mathbb Z/3\rightarrow B\F$. Then $f$ factors through the cofiber $C$ of $ev$ and, since $(B\Sigma_3)^\wedge_3$ is $B\mathbb Z/3$-null, we obtain a factorization of $f$, $f'\colon P_{\Sigma B\mathbb Z/3}(C)\rightarrow (B\Sigma_3)^\wedge_3$ such that the following diagram is homotopy commutative
\[
\begin{tikzcd}
\Linp \arrow[r, Rightarrow, no head] \arrow[d] & \Linp \arrow[d, "\p{r}{3}"] \\
P_{\bz3}(C)  \arrow[r, "f'"]                   & \p{(B\Sigma_3)}{3}.
\end{tikzcd}
\]

The strategy is to construct a homotopy inverse of $f'$, $\bar{\Theta}\colon \p{(B\Sigma_3)}{3}\rightarrow \p{P_{\bz3}(C)}{3}$, up to $3$-completion, which fits in the previous diagram up to homotopy. 

Since $\Sigma_3$ has an abelian normal $3$-Sylow subgroup $\mathbb Z/3$, we have that ${(B\mathbb Z/3)_{h\mathbb Z/2}\rightarrow B\Sigma_3}$ is an equivalence. Consider the fiber sequence $BN\rightarrow BS\rightarrow B\mathbb Z/3$ and the morphism ${\p{r}{3}|_{BS}\colon BS\rightarrow \p{P_{\bz3}(C)}{3}}$, by Zabrodsky's lemma 
\cite[Proposition 3.4]{MR1397728}, $\p{r}{3}|_{BS}$ factors (uniquely up to homotopy) via $\Theta'\colon B\mathbb Z/3\rightarrow \p{P_{\bz3}(C)}{3}$. In order to get $\bar{\Theta}$, we only need check that $\Theta'$ is $\mathbb Z/2$-equivariant up to homotopy. For any $\F$ in the hypothesis of the proposition, note that there is an element in $\omega'\in \Out_\F(S)$ which project to $\omega\in \Out_{\Sigma_3}(\mathbb Z/3)$ (they are called $\eta$ or $\omega$ in the tables \cite[Theorem 5.10]{MR2272147}). Since $\Theta'$ is unique up to homotopy factoring $\p{r}{3}\circ \Theta$, and $\Theta\circ \omega'\simeq \Theta$, it follows $\omega\circ \Theta'\simeq \Theta'$.

Next we check that $\bar{\Theta}$ is a homotopy inverse to $f$. First consider the following homotopy commutative diagram:

\[
\begin{tikzcd}
CW_{B\mathbb Z/3}(B{\F}) \arrow[d]                   &                                                &                                    \\
\Linp  \arrow[d, "r"] \arrow[r,Rightarrow, no head] & \Linp \arrow[d] \arrow[r, Rightarrow, no head] & \Linp \arrow[d, "\p{r}{3}"]        \\
P_{\Sigma B\mathbb Z/3}(C)  \arrow[r, "f'"]          & \p{(B\Sigma_3)}{3}  \arrow[r, "\bar{\Theta}"]  & \p{P_{\Sigma B\mathbb Z/3}(C)}{3}.
\end{tikzcd}
\]

Since $3$-completion on the bottom line also gives a homotopy commutative diagram, uniqueness on Zabrodsky's lemma (see 
\cite[Proposition 3.4]{MR1397728}) shows that $\bar{\Theta}\circ f'$ is $3$-completion. So $\bar{\Theta}\circ \p{(f')}{3}\simeq id$.

Now consider the following homotopy commutative diagram 

\[
\begin{tikzcd}
BK \arrow[r] \arrow[d]             & F \arrow[d]                                                   &                      \\
BS \arrow[r, "\Theta"] \arrow[d]   & \Linp \arrow[d] \arrow[r, Rightarrow, no head]                & \Linp \arrow[d, "f"] \\
B\mathbb Z/3    \arrow[r, "\iota"] & \p{(B\Sigma_3)}{3} \arrow[r, "\p{(f')}{3}\circ \bar{\Theta}"] & \p{(B\Sigma_3)}{3}  
\end{tikzcd}
\]

The map $\p{(f')}{3}\circ \bar{\Theta}$ is determined by its restriction to the Sylow $3$-subgroup $\mathbb Z/3$ by 
\cref{l:exotic higherlimits}. Again $\iota\circ\p{(f')}{3}\circ \bar{\Theta}$ and $\iota$ give homotopy commutative diagrams when placed in the bottom line, then uniqueness on Zabrodsky's lemma (see 
\cite[Proposition 3.4]{MR1397728}) shows that they are homotopic.
\end{proof} 

\bibliographystyle{alpha}
\bibliography{Bibliografia}

\end{document}